\documentclass[preprint,12pt]{elsarticle}
\usepackage{amsmath,amssymb}
\usepackage{amsthm}
\usepackage{comment}
\usepackage{mathrsfs}
\usepackage{braket}
\usepackage{amscd} 

\newtheorem{thm}{Theorem}[section]

\newtheorem{prop}[thm]{Proposition}
\theoremstyle{definition}
\newtheorem{dfn}[thm]{Definition}
\newtheorem{ex}[thm]{Example}

\newtheorem{lem}[thm]{Lemma}

\theoremstyle{remark}
\newtheorem{rem}[thm]{Remark}

\theoremstyle{Matsuda's theorem}

\newcommand{\Z}{\mathbb{Z}} 
\newcommand{\B}[1]{\Omega(#1)}
\newcommand{\D}{\mathcal{D}} 
\renewcommand{\P}{\mathcal{P}} 
\newcommand{\sub}[1]{{\rm Sub}({#1})}
\newcommand{\ns}[1]{{\rm NS}({#1})}
\newcommand{\bs}[2]{{\rm S}_{#1}({#2})}

\begin{document}

\begin{frontmatter}
	
	
	
	\title{On the unit groups of partial Burnside rings}
	
	
	\author{Masahiro Wakatake}
	
	\address{Department of Mathematics, Kindai University, 3-4-1 Kowakae, Higashiosaka City, Osaka 577-8502, Japan}
	
	\begin{abstract}
		We describe the generalized Matsuda's theorem in \cite{ma82}, and some results of a Burnside ring extend a partial Burnside ring. In particular, we give isomorphism between partial Burnside rings of different groups. Moreover, we consider the relationship between image of Frobenius-Wielandt homomorphism, a partial Burnside ring, and structure of a group.
		
	\end{abstract}
	
	
\end{frontmatter}
\section{Introduction}
Let $G$ be a finite group. The Burnside ring of a finite group $G$ is defined to be the Grothendieck ring of the semi-ring generated by isomorphism classes of finite (left) $G$-sets where the addition and multiplication are given by disjoint unions and Cartesian products. Denote by $\B{G}$ the Burnside ring of $G$. If a family $\D$ of subgroups of $G$ contains $G$ and it is closed under taking conjugation and intersection, then $\D$ is called a collection of $G$. We call the Grothendieck ring of the semi-ring generated by isomorphism classes of finite (left) $G$-sets whose stabilizers of any element lie in a collection $\D$ of $G$ a partial Burnside ring (PBR for shot) relative to $\D$ of $G$. We denote the PBR relative to a collection $\D$ of $G$ by $\B{G,\D}$.

Results of this paper is the following :

\vspace{5pt}
\noindent{\bf Theorem \ref{teiri1}.} 
	Let $G_1$ and $G_2$ be finite groups, and $\D$ a collection of $G_1$. If $f:G_1\longrightarrow G_2$ is a surjective group homomorphism, then
	\[
	\B{G_1,\D_{\ker f}} \simeq \B{G_2,f(\D_{\ker f})} ,
	\]
	where $f(\D_{\ker f})=\{f(H)\leq G_2\ |\ H\in\D,\ \ker f\leq H\}$ and $\D_{\ker f}=\{H\in\D\ |\ \ker f\leq H\}$.
\vspace{5pt}

\vspace{5pt}
\noindent{\bf Theorem \ref{oddeq}.} 
		Let $G$ be a finite group. Then the following are equivalent.
		\begin{itemize}
			\item[$(1)$] $G$ is non-abelian simple group, then $|G|=$even.
			\item[$(2)$] $|G|=$odd, then $|\B{G}^\times|=2$.
			\item[$(3)$] $|G|=$odd, then $\B{G}^\times\simeq \B{G,\ns{G}}^\times$.
			\item[$(4)$] $|G|=$odd, then $\B{G}^\times\simeq \B{G, \sub{G}_{G'}}^\times$.
		\end{itemize}
		Here, $\ns{G}$ is the set of all normal subgroups of $G$, $G'$ is a commutator subgroup of $G$, and $\sub{G}_{G'}=\{H\leq G\ |\ G'\leq H\}$.
\vspace{5pt}

\noindent{\bf Definition \ref{seminil}.} 
	Let $p$ be a any prime number. For any positive integer $a$, we call a finite group $G$ $p^a$-\emph{seminilpotent} if for any subgroup $K$ of $G$ has index which is divisible by $p^a$, there exists a normal subgroup $N$ of $G$ such that $K\leq N$ and $|G:N|=p^a$, and $\#\{ N \trianglelefteq G \mid K\leq N,\ |G:N|=p^a \}\equiv 1$ mod $p$.
\vspace{5pt}

\noindent{\bf Theorem \ref{imgfw}.} 
	Let $G$ be a finite group of even order $n$, and $C$ be a cyclic group of the same order $n$. if $G$ is $2$-seminilpotent, then 
	\[
	\alpha(\B{C}^\times)\subseteq \B{G,\ns{G}}^\times.
	\]
	In particular, 
	\[
	\alpha(\B{C}^\times) =\braket{-1_{\B{G}},\ \prod_{\{N\leq G;|G:N|=2\}}u_H},
	\]
	where $u_N=1_{\B{G}}-[G/H]$.
\vspace{5pt}

The paper is organized as follows. In Section 2, we recall the basic definitions and results from PBR. In Section 3, we describe the generalized Matsuda's theorem in \cite{ma82}. In Section 4, we show some examples of Matsuda's theorem. In Section 5, we show more results associated with Matsuda's theorem. In Section 6, we consider the relationship between image of Frobenius-Wielandt homomorphism, a partial Burnside ring, and structure of a group.

\section{Preliminaries}
\subsection{Notation}
Let $G$ be a finite group. Denote by $\sub{G}$ the set of subgroups of $G$. Denote by $\ns{G}$ the set of normal subgroups of $G$. For a family $\D$ of subgroups of $G$ with closed under $G$-conjugation, write $\D^c$ for the set of the conjugacy classes of $\D$. Denote by $(H)$ the set of all $G$-conjugate subgroups of a subgroup $H$ of $G$. Denote by $[X]$ the isomorphism class of finite $G$-set $X$. If $X$ is a finite set, write $|X|$ or $\#X$ for the cardinality of $X$. 

\subsection{Partial Burnside rings}
Let $G$ be a finite group. Then the Burnside ring $\B{G}$ of $G$ can be regarded as a free abelian group with basis $\set{[G/H]|(H)\in\sub{G}^c}$. The multiplication in the ring is given by 
\[
[G/H]\cdot[G/K]=\sum_{HgK\in H\backslash G/K} [G/(H\cap gKg^{-1})].
\]
\begin{dfn}
	Let $G$ be a finite group. The family $\D$ of subgroups of $G$ is called a \emph{closed} if $\D$ satisfies the following $2$ conditions:
	\begin{itemize}
		\item $H\cap K\in\D$ for any $H,K\in\D$.
		\item $gHg^{-1}\in\D$ for any $H\in\D$ and $g\in G$.
	\end{itemize}
In particular, we call a closed family $\D$ of subgroups $G$ \emph{collection} if $G\in\D$.
\end{dfn}
For a closed family $\D$ of subgroups of $G$, we put 
\[
\B{G,\D}:= \braket{[G/H]\ |\ (H)\in\D^c}_\Z.
\]
In particular, we call $\B{G,\D}$ a {\em partial Burnside ring} (PBR for short) {\em relative to $\D$ of $G$} if $\D$ is a collection. 
\begin{rem}
	For a closed family $\D$ of subgroups of a finite group $G$, A $\B{G,\D}$ is a subring of $\B{G}$ (it need not unital). A PBR $\B{G,\D}$ is a unitary subring of $\B{G}$.
\end{rem}

For $K\leq G$ and a $G$-set $X$, we set
\[
{\rm inv}_K(X):=\{x\in X|kx=x \text{\ for all\ } k\in K\}.
\]
If $\D$ is a closed family of subgroups of $G$, then for each $K\in\D$, the $\Z$-linear map $\varphi^\D_K:\B{G,\D}\rightarrow\Z$ which is induced by $[G/H]\mapsto \#{\rm inv}_K(G/H)$ is a ring homomorphism. By \cite[Theorem 3.10]{yo90}, the map $\varphi^\D=(\varphi^G_K)_{(K)\in\D^c} : \B{G,\D}\rightarrow \prod_{(K)\in\D^c}\Z$ is an injective ring homomorphism, where $\varphi^\D$ is called a \emph{Burnside homomorphism relative to $\D$ of $G$}. We call a matrix $M^\D:=({\rm inv}_K(G/H))_{(K),(H)\in\D^c}$ a \emph{table of marks relative to $\D$ of $G$}.

\section{A generalization of Matsuda's theorem}
\begin{dfn}
	Let $G$ be a finite group. A collection $S$ of $G$ is called a \emph{basic collection} if $S$ satisfies the following three conditions:
	\begin{itemize}
		\item $\braket{1}\in S$.
		\item If $H\in S$, then $H$ is a normal subgroup of $G$.
		\item If $H,K\in S$, then $HK\in S$,
	\end{itemize}
where $HK$ is a (normal) subgroups of $G$ generated by $H$ and $K$.
\end{dfn}

	Let $G$ be a finite group, $\D$ a collection of $G$, and $S$ a basic collection of $G$. For any $H\in S$, we put
	\[
	\bs{\D}{H}:=\{F\in\D\ |\ H\leq F,\text{\ if\ } H'\in S \text{ and } H\leq H'\leq F,\ \text{then}\ H=H'  \}.
	\]
Obviously, $\bs{\D}{H}$ is a closed family of subgroups of $G$. Therefore for each $H\in S$ $\B{G,\bs{\D}{H}}$ is a subring of $\B{G}$, and the map $\varphi^{\bs{\D}{H}}$ is an injective ring homomorphism. We remark that if $\bs{\D}{G}$ is empty set, then $\B{G,\bs{\D}{H}}=\{0\}$.

For a collection $\D$ of $G$ and a basic collection $S$ of $G$, we put
\[
S^{\D}:=\{H\in S|\bs{\D}{H}\neq\emptyset\}.
\]
\begin{lem}\label{matsudahodai}
	Let $G$ be a finite group, $\D$ a collection of $G$, and $S$ be a basic collection of $G$. 
	\begin{itemize}
		\item[$(1)$] $\B{G,\D}=\bigoplus_{H\in S^\D} \B{G,\bs{\D}{H}}$.
		\item[$(2)$] $\B{G,\bs{\D}{H}}\cdot\B{G,\bs{\D}{F}}\subseteq\B{G,\bs{\D}{H\cap F}}$ for any $H,F\in S^\D$.
		\item[$(3)$] $1_{\B{G}}\in\B{G,\bs{\D}{G}}=\Z\cdot1_{\B{G}}$.
	\end{itemize}
\end{lem}
\begin{proof}
	For $H_1,H_2\in S^\D$, if $F\in \bs{\D}{H_1}\cap\bs{\D}{H_2}$, then $H_1=H_2$ because $H_1\leq H_1H_2\leq F$. Therefore if $H_1\neq H_2$, then $\bs{\D}{H_1}\cap\bs{\D}{H_2}=\emptyset$. For any $F\in\D$, the set $\{H\in S | H\leq F\}$ is non-empty since $\braket{1}\in S$. Let $H_0$ be a maximal element of the set $\{H\in S | H\leq F\}$. Then $F\in\bs{\D}{H_0}$. Therefore $\D=\bigsqcup_{H\in S^\D}\bs{\D}{H}$, we obtain $(1)$. Obviously, for $F_1\in\bs{\D}{H_1},F_2\in\bs{\D}{H_2}$, $H_1\cap H_2\leq gF_1g^{-1}\cap F_2$. If there exist $H'\in S$ such that $H_1\cap H_2\leq H'\leq gF_1g^{-1}\cap F_2$, then $H'\leq H_1\cap H_2$ because $H_i\leq H'H_i\leq F_i$ for $i=1,2$. Therefore we obtain $(2)$. $(3)$ is trivial.
\end{proof}

\begin{thm}[{\cite[Theorem $4.1$, $4.2$]{ma82}}]\label{mdt1}
	Let $G$ be a finite group, $\D$ be a collection of $G$, and $S$ be a basic collection of $G$. Then
	\[
	|\B{G,\D}^\times|=2\prod_{H\in S\backslash\{G\}}|\overline{ \B{G,\bs{\D}{H}} }|,
	\]
	where 
	\[
	\overline{ \B{G,\bs{\D}{H}} }:=\{x\in\B{G,\bs{\D}{H}}\ |\ x^2+2x=0\}.
	\]
	Moreover,
	\[
	|\B{G,\D}^\times|=2\prod_{H\in S\backslash\{G\}}  \left|(M^{\bs{\D}{H}})^{-1}\left(\Z'_2{}^{|\bs{\D}{H}|}\right)\cap\Z^{|\bs{\D}{H}|}\right| ,
	\]
	where $M^{\bs{\D}{H}}$ is a table of marks relative to $\bs{\D}{H}$ of $G$, and $\Z'_2=\{0,-2\}\subseteq\Z$.
	
	\noindent In particular,
	\[
	\{-1_{\B{G}}\}\cup\bigcup_{H\in S}\bigcup_{x\in \overline{ \B{G,\bs{\D}{H}} }\backslash \{0\}} \{1_{\B{G}}+x\}
	\]
	is a set of generators of $\B{G,\D}$.
\end{thm}
\begin{proof}
	By Lemma $\ref{matsudahodai}$ and \cite[Lemma $3.5$]{ma82}, we have this theorem in the case of $S=S^\D$. For any $H\in S\backslash S^\D$, $|\overline{ \B{G,\bs{\D}{H}} }|=1$ since $\B{G,\bs{\D}{H}}=\{0\}$, completing the proof.
\end{proof}
\section{Examples of Matsuda's theorem}

\begin{prop}\label{seiki}
	Let $G$ be a finite group, and let $\ns{G}$ is the set of all normal subgroups of $G$. Then
	\[
	|\B{G,\ns{G}}^\times|=2^{\#\{H\leq G | |G:H|\leq2\}}.
	\]
	Moreover,
	\[
	\{-1_{\B{G}}\}\cup\{u_H \ |\ H\leq G,\ |G:H|=2\}
	\]
	is a set of generators of $\B{G,\ns{G}}^\times$ where $u_H=1_{\B{G}}-[G/H]$. 
\end{prop}
\begin{proof}
	Now, $S:=\ns{G}$ is a basic collection. For $H\in\ns{G}\backslash{G}$, $\B{G,S_{\ns{G}}(H)}=\Z\cdot[G/H]$ by $S_{\ns{G}}(H)=\{H\}$. Therefore
	\[
	|\overline{ \B{G,\bs{\ns{G}}{H}} }|=
	\begin{cases}
	1 & |G:H|\neq2,\\
	2 & |G:H|=2,
	\end{cases}
	\]
since $H$ is normal subgroup of $G$ .This completes the proof.
\end{proof}

\begin{thm}\label{oddeq}
	Let $G$ be a finite group. Then the following are equivalent.
	\begin{itemize}
		\item[$(1)$] $G$ is non-abelian simple group, then $|G|=$even.
		\item[$(2)$] $|G|=$odd, then $|\B{G}^\times|=2$.
		\item[$(3)$] $|G|=$odd, then $\B{G}^\times\simeq \B{G,\ns{G}}^\times$.
		\item[$(4)$] $|G|=$odd, then $\B{G}^\times\simeq \B{G, \sub{G}_{G'}}^\times$.
	\end{itemize}
	Here, $\ns{G}$ is the set of all normal subgroups of $G$, $G'$ is a commutator subgroup of $G$, and $\sub{G}_{G'}=\{H\leq G\ |\ G'\leq H\}$.
\end{thm}
\begin{proof}
	(1) $\Leftrightarrow$ (2) by \cite{so76}. By Proposition $\ref{seiki}$, $|\B{G,\ns{G}}|=2$ since $|G|$ is odd. Hence (2) $\Leftrightarrow$ (3). By Lemma $\ref{zyouyo}$, $\B{G, G_{G'}}\simeq\B{G/G'}$. Since $G/G'$ is an abelian group, $|\B{G/G'}|=2$ by Proposition $\ref{seiki}$. Therefore (2) $\Leftrightarrow$ (4). This completes the proof.
\end{proof}

\begin{dfn}
	Let $W$ be a finite Coxeter group with Coxeter system $(W,S)$. Then a subgroup $P$ is called parabolic subgroup if there exists $J\subseteq S$ such that $(P)=(\braket{J})$. Denoted by $\P_W$ the set of all parabolic subgroups of $W$.
\end{dfn}
Let $W$ be a finite Coxeter group. Then the set $\P_W$ of all parabolic subgroups of $W$ becomes a collection (see \cite{so76}).

\begin{ex}
	Let $W$ be a Coxeter group of type $I_2(m)$. Then
	\[
	|\B{W,\P_W}^\times|=4.
	\]
\end{ex}
\begin{proof}
	We put $G:=D_{2m}=\braket{\sigma,\tau|\sigma^m=\tau^2=(\sigma\tau)^2=1}$, and $L=\{\tau, \sigma\tau\}$. Then $G$ is a Coxeter group of type $I_2(m)$ with Coxeter system $(G,L)$. 
	Let $S=\{ \braket{1}, G \}$. Then $S$ is a basic collection of $G$. Since
	\[
	\P_G{}^c =
	\begin{cases}
	\{(\braket{1}), (\braket{\tau}), (\braket{\sigma\tau}), (G) \} & m \text{ is even},\\
	\{(\braket{1}), (\braket{\tau}), (G) \} & m \text{ is odd},
	\end{cases}
	\]
	it follows that
	\[
	\bs{\P_G}{\braket{1}}^c=
	\begin{cases}
	\{(\braket{1}), (\braket{\tau}), (\braket{\sigma\tau}) \} & m \text{ is even},\\
	\{(\braket{1}), (\braket{\tau}) \} & m \text{ is odd}.
	\end{cases}
	\]
	The table of marks relative to $\bs{\P_G}{\braket{1}}$ of $G$ is as follows
	\[
	M:=M^{\bs{\P_G}{\braket{1}}}= 
	\begin{cases}
	
		\left(
		\begin{array}{ccc}
		2m & m & m \\
		0  & 2 & 0 \\
		0  & 0 & 2
		\end{array}
		\right) & m \text{ is even},\\[8mm]
	
	\left(
	\begin{array}{cc}
	2m & m  \\
	0  & 1  \\
	\end{array}
	\right) & m \text{ is odd}.\\
	\end{cases}
	\]
	Therefore
	\[
	M^{-1}\left(\Z'_2{}^{|\bs{\P_G}{\braket{1}}|}\right)\cap\Z^{|\bs{\P_G}{\braket{1}}|}=
	\begin{cases}
	\left\{
	\left(
	\begin{array}{c}
	0  \\
	0  \\
	0  
	\end{array}
	\right),
	\left(
	\begin{array}{c}
	1  \\
	-1  \\
	-1  
	\end{array}
	\right)
	\right\}
	 & m \text{ is even},\\[8mm]
	
		\left\{
	\left(
	\begin{array}{c}
	0  \\
	0  \\
	\end{array}
	\right),
	\left(
	\begin{array}{c}
	1  \\
	-2   
	\end{array}
	\right)
	\right\} & m \text{ is odd}.\\
	\end{cases}
	\]
This completes the proof.
\end{proof}

\begin{ex}\label{hanrei}
	Let $G=A_4$ be a alternating group of degree $4$. Then
	\[
	|\B{G}^\times|=4.
	\]
\end{ex}
\begin{proof}
	Let $S=\{\braket{1},\ V_4,\ G\}$ where $V_4$ is a Klein four-group.
	\[
	\bs{\sub{G}}{\braket{1}}=\{(\braket{1}),\ (\braket{H_1}),\ \braket{H_2}\},
	\]
	where $H_1=\{1, (1\ 2)(3\ 4)\}$, and $H_2=\{1, (1\ 2\ 3), (1\ 3\ 2)\}$. 
	
	The table of marks relative to $\bs{\sub{G}}{\braket{1}}$ of $G$ is as follows
	\[
	M:=M^{\bs{\sub{G}}{\braket{1}}}= 
		\left(
		\begin{array}{ccc}
			12 & 6 & 4 \\
			0  & 2 & 0 \\
			0  & 0 & 1
		\end{array}
		\right) .
	\]
	Therefore
	\[
	M^{-1}\left(\Z'_2{}^{|\bs{\sub{G}}{\braket{1}}|}\right)\cap\Z^{|\bs{\sub{G}}{\braket{1}}|}=
	\left\{
	\left(
	\begin{array}{c}
	0  \\
	0  \\
	0  
	\end{array}
	\right),
	\left(
	\begin{array}{c}
	1  \\
	-1  \\
	-2  
	\end{array}
	\right)
	\right\}.
	\]
\end{proof}

\section{Further results}
For finite group $G$, collection $\D$ of $G$, and subgroup $N$ of $G$, we put 
\[
\D_N=\{H\in\D\ |\ N\leq H\}.
\]
It is easy to see that if $N$ is a normal subgroup of $G$, then $\D_N$ is collection of $G$.
\begin{lem}\label{zyouyo}
	Let $G$ be a finite group, and let $N$ be a normal subgroup of $G$. Then
	\[
	\B{G/N}\simeq\B{G,\sub{G}_{N}},
	\]
	where $\sub{G}_N:=\{H\leq G\ |\ N\leq H\}$.
\end{lem}
\begin{proof}
	By correspondence theorem, the map $\pi: G_N\longrightarrow \sub{G/N}$ defined by $H\longmapsto H/N$ is a bijection. Hence
	\[
	\B{G/N}\simeq\B{G,G_{N}}
	\]
	as $\Z$-module. Therefore it suffices to show that
	\[
	[\pi(G)/\pi(H)]\cdot[\pi(G)/\pi(K)]=\sum_{HxK\in H\backslash G/K}[\pi(G)/\pi(gHg^{-1}\cap K)].
	\]
	For $H,K\in G_N$, we put the map $g:H\backslash G/K\longrightarrow \pi(H)\backslash \pi(G)/\pi(K); HxK\longmapsto \pi(H)xN\pi(K)$. Since
	\begin{eqnarray*}
		HxK=HyK&\Longleftrightarrow&yN=hxkN\quad ({}^\exists h\in H, {}^\exists k\in K)\\
		&\Longleftrightarrow& \pi(H)yN\pi(K)=\pi(H)xN\pi(K)\\
	\end{eqnarray*}
	the map $g$ is injective. The map $g$ is clearly surjective hence $g$ is bijection. Since $\pi(gHg^{-1})=gN\pi(H)g^{-1}N$, 
	\begin{align*}
	\sum_{HxK\in H\backslash G/K}[\pi(G)//\pi(gHg^{-1}\cap K)]=&\sum_{HxK\in H\backslash G/K}[\pi(G)//\pi(gHg^{-1})\cap \pi(K)]\\
	=&\sum_{HxK\in H\backslash G/K}\hspace{-0.5cm}[\pi(G)//gN\pi(H)g^{-1}N\cap \pi(K)]\\
	=&\hspace{-1.2cm}\sum_{\pi(H)xN\pi(K)\in \pi(H)\backslash \pi(G)/\pi(K)}\hspace{-1.6cm} [\pi(G)//gN\pi(H)g^{-1}N\cap \pi(K)]\\
	=&[\pi(G)/\pi(H)]\cdot[\pi(G)/\pi(K)].
	\end{align*}
	This completes the proof.
\end{proof}
\begin{lem}
	Let $G_1$ and $G_2$ be finite groups, and $\D$ a collection of $G_1$. If $f:G_1\longrightarrow G_2$ is a surjective group homomorphism, then the family
	\[
	f(\D_{\ker f})=\{f(H)\leq G_2\ |\ H\in\D,\ \ker f\leq H\}
	\]
	is a collection of $G_2$
\end{lem}
\begin{proof}
	Since the map $f$ is a surjective group homomorphism, $G_2\in f(\D)$ and $g_2H_2g_2^{-1}\in f(\D_{\ker f})$ for any $H_2\in f(\D_{\ker f})$, and $g_2\in G_2$. Therefore it suffices to show that $f(H\cap K)=f(H)\cap f(K)$ for any $H,K\in \D$. For any $y\in f(H)\cap f(K)$, there exist $h\in H$ and $k\in K$ such that $y=f(h)=f(k)$. Hence $k\in h\ker f \subset H\ker f = H$ since $\ker f \leq H$. Therefore $y\in f(H\cap K)$, completing the proof.
\end{proof}

\begin{thm}\label{teiri1}
	Let $G_1$ and $G_2$ be finite groups, and $\D$ a collection of $G_1$. If $f:G_1\longrightarrow G_2$ is a surjective group homomorphism, then
	\[
	\B{G_1,\D_{\ker f}} \simeq \B{G_2,f(\D_{\ker f})} ,
	\]
	where $f(\D_{\ker f})=\{f(H)\leq G_2\ |\ H\in\D,\ \ker f\leq H\}$ and $\D_{\ker f}=\{H\in\D\ |\ \ker f\leq H\}$. 
\end{thm}

By Theorem $\ref{teiri1}$, we have the following Matsuda's theorem.
\begin{thm}[{\cite[Theorem $4.4$]{ma82}}]
	Let $G_1$ and $G_2$ be finite groups, and $\D$ a collection of $G_1$. If $f:G_1\longrightarrow G_2$ is a surjective group homomorphism, then
	\[
	|\B{G_1,\D}^\times| = |\B{G_2, f(\D_{\ker f})}^\times|\cdot |\overline{ \B{G_1,\bs{\D}{H}} }|,
	\]
	where $S=\{\braket{1},\ \ker f,\ G\}$, and $f(\D_{\ker f})=\{f(H)\ |\ H\in\D,\ \ker f\leq H\}$.
\end{thm}
\begin{proof}
	By Theorem $\ref{mdt1}$, we have
	\[
	|\B{G,\D}^\times| = 2 |\overline{ \B{G,\bs{\D}{\ker f}} }|\cdot|\overline{ \B{G,\bs{\D}{H}} }|.
	\]
	$|\B{G_2,f(\D_{\ker f})}^\times|= |\B{G,\D_{\ker f}}^\times|=2 |\overline{ \B{G,\bs{\D}{\ker f}} }$ by Theorem $\ref{teiri1}$. This completes the proof.
\end{proof}

\section{Application of a partial Burnside Rings}

\begin{dfn}\label{seminil}
	Let $p$ be a any prime number. For any positive integer $a$, we call a finite group $G$ $p^a$-\emph{seminilpotent} if for any subgroup $K$ of $G$ has index which is divisible by $p^a$, there exists a normal subgroup $N$ of $G$ such that $K\leq N$ and $|G:N|=p^a$, and $\#\{ N \trianglelefteq G \mid K\leq N,\ |G:N|=p^a \}\equiv 1$ mod $p$.
\end{dfn}

\begin{rem}
	seminilpotent group is a nilpotent group. However nilpotent group is not seminilpotent in general. For example let $H$ be a non-nilpotent group of odd order, and let $E$ be a $2$-group. Then $G:=H\times E$ is not nilpotent but $2$-seminilpotent.
\end{rem}

Let $G$ be a finite group. For any subgroup $K$ of $G$, we put
\[
\overline{K}:=\bigcap_{\{N \trianglelefteq G \mid K\leq N\}} N,
\] 

\begin{prop} 
	Let $G$ be a finite group. If $K$ be a subgroup of $G$, then
	\[
	\begin{split}
	\{ N \trianglelefteq G \mid K\leq N,\ |G:N|=p^a \} 
	=\{  N \trianglelefteq G \mid \overline{K}\leq N, \ |G:N|=p^a \}.
	\end{split}
	\]
\end{prop}
\begin{proof}
	Let $n$ be a positive integer, and $\lambda$ a primitive $n$-th root of unity. For each $N\in\{ N \trianglelefteq G \mid |G:N|=p^a \}$, we put
	\[
	u_N(\lambda):=[G/G]+\frac{\lambda-1}{p^a}[G/N] \in \mathbb{C}\otimes_\mathbb{Z}\B{G, \ns{G}}.
	\]
	For any $K\leq G$, since
	\[
	\varphi_K(u_N(\lambda))=
	\begin{cases}
	\lambda & K\leq N,\\
	1 & \text{otherwise},
	\end{cases}
	\]
	we have
	\begin{eqnarray*}
	\varphi_K(\prod_{N\in\{ N \trianglelefteq G \mid |G:N|=p^a \}}u_N(\lambda))&=&\lambda^{\#\{ N \trianglelefteq G \mid K\leq N,\ |G:N|=p^a \}}\\
	&=&\varphi_{\overline{K}}(\prod_{N\in\{ N \trianglelefteq G \mid |G:N|=p^a \}}u_N(\lambda))\\
	&=&\lambda^{\#\{ N \trianglelefteq G \mid \overline{K}\leq N,\ |G:N|=p^a \}}.
	\end{eqnarray*}
	Therefore 
	\[
	\begin{split}
	\#\{ N \trianglelefteq G \mid K\leq N,&\ |G:N|=p^a \} \\
	&\equiv\{  N \trianglelefteq G \mid \overline{K}\leq N, \ |G:N|=p^a \} \quad(\text{mod }n).
	\end{split}
	\]
	This completes the proof.
\end{proof}
\begin{proof}[Easy proof by group theory]
	
	Obviously, 
	\[
	\{ N \trianglelefteq G \mid K\leq N,\ |G:N|=p^a \} \supseteq \{  N \trianglelefteq G \mid \overline{K}\leq N, \ |G:N|=p^a \}.
	\]
	We take any $N\in\{ N \trianglelefteq G \mid K\leq N,\ |G:N|=p^a \}$. Since $N\cap \overline{K}$ is a normal subgroup of $G$ and $K\leq N\cap \overline{K}\leq \overline{K}$, $N\cap \overline{K}=\overline{K}$. Hence $\overline{K}\leq N$. Therefore 
	\[
	\{ N \trianglelefteq G \mid K\leq N,\ |G:N|=p^a \} = \{  N \trianglelefteq G \mid \overline{K}\leq N, \ |G:N|=p^a \}.
	\]
\end{proof}

We review a theorem of \cite{dsy92}.
\begin{thm}[{\cite[Theorem $1.$]{dsy92}}]
	Let $C_n$ be a cyclic group of order $n$, and let $C:=C_{|G|}$. Then the map 
	\[
	\alpha=\alpha^G : \B{C}\longrightarrow \B{G}
	\]
	given by
	$\varphi^G_K(\alpha(x))=\varphi^C_{C_{|K|}}(x)$ for each subgroup $K$ of $G$
	is a ring homomorphism. The map $\alpha$ is called the \emph{Frobenius-Wielandt homomorphism} 
\end{thm}

\begin{thm}\label{imgfw}
	Let $G$ be a finite group of even order $n$, and $C$ be a cyclic group of the same order $n$. if $G$ is $2$-seminilpotent, then 
	\[
	\alpha(\B{C}^\times)\subseteq \B{G,\ns{G}}^\times.
	\]
	In particular, 
	\[
	\alpha(\B{C}^\times) =\braket{-1_{\B{G}},\ \prod_{\{N\leq G;|G:N|=2\}}u_H},
	\]
	where $u_N=1_{\B{G}}-[G/H]$.
\end{thm}
\begin{proof}
	By Proposition $\ref{seiki}$, $\B{C}^\times=\braket{-1_{\B{C}}, 1_{\B{C}}-[C/C_{n/2}]}$. Since
	\[
	\varphi^G_K(\alpha(1_{\B{C}}-[C/C_{n/2}]))=\varphi^C_{C_{|K|}} (1_{\B{C}}-[C/C_{n/2}])=
	\begin{cases}
	-1& 2|K|\text{ divides }n,\\
	1&\text{othewise},
	\end{cases}
	\]
	it follows that
	\[
	\varphi^G_K(\alpha(1_{\B{C}}-[C/C_{n/2}]))=
	\begin{cases}
	-1 & |G:K| \text{ is even},\\
	1 & \text{othewise}.
	\end{cases}
	\]
	For each $H\leq G$ has index $2$ in $G$, 
	\[
	\varphi_K(u_N)=
	\begin{cases}
	-1 & K\leq N,\\
	1 & \text{otherwise}.
	\end{cases}
	\]
	where $u_H=[G/G]-[G/H]$. Since $G$ is a $2$-seminilpotent, for any $K\leq G$ has even index in $G$, there exist $N\leq G$ such that $|G:N|=2$. Therefore
	\begin{eqnarray*}
	\varphi^G_K(\prod_{\{N\leq G;|G:N|=2\}}u_N)&=&(-1)^{\{N\leq G \mid K\leq N, |G:N|=2\}}\\
	&=&
	\begin{cases}
		-1 & |G:K| \text{ is even},\\
		1 & \text{othewise}.
	\end{cases}
	\end{eqnarray*}

	 By injectivity of the Burnside homomorphism, completing the proof.

\end{proof}

\begin{rem}
	If $G$ is a finite group of odd order $n$ and $C$ is a cyclic group of the same order $n$, then 
	$\B{C}^\times=\braket{-1_{\B{C}}}$ by Proposition $\ref{seiki}$. Therefore $\alpha(\B{C}^\times)\subseteq \B{G,\ns{G}}^\times$
\end{rem}
\begin{rem}	
	In general, for any finite group $G$ and $C=C_{|G|}$, it is not necessarily $\alpha(\B{C}^\times)\subseteq \B{G,\ns{G}}^\times$ (see Example $\ref{hanrei}$).
\end{rem}

\end{document}